\documentclass{amsart}

\usepackage[utf8]{inputenc}
\usepackage{amscd}
\usepackage{amsmath}
\usepackage{amssymb}
\usepackage{amsthm}
\usepackage{epsf}
\usepackage{verbatim}
\usepackage{graphicx}
\usepackage{placeins}
\usepackage{enumitem}
\input epsf.tex

\usepackage[letterpaper,top=3cm,bottom=2cm,left=3cm,right=3cm,marginparwidth=1.75cm]{geometry}

\newtheorem{theorem}{Theorem}
\newtheorem{lemma}[theorem]{Lemma}

\newtheorem{exercise}[theorem]{Exercise}
\newtheorem{example}{Example}[section]

\let\oldexercise\exercise
\renewcommand{\exercise}{\oldexercise\normalfont}

\let\oldexample\example
\renewcommand{\example}{\oldexample\normalfont}

\newcommand{\R}{\mathbb{R}}

\title{Triple Crossing Number and Double Crossing Braid Index}
\date{May 8, 2018}
\author{Daishiro Nishida}

\begin{document}

\begin{abstract}
Traditionally, knot theorists have considered projections of knots where there are two strands meeting at every crossing. A triple crossing is a crossing where three strands meet at a single point, such that each strand bisects the crossing. In this paper we find a relationship between the triple crossing number and the double crossing braid index, namely $\beta_2(L) \le c_3(L) + 1$. We find an infinite family of knots that achieve equality, which allows us to determine both the double crossing braid index and the triple crossing number of these knots.
\end{abstract}

\maketitle


\section{Introduction}
\label{sec:intro}

In traditional knot theory, knots are drawn in a projection where there are two strands passing over each other at every crossing. An \emph{$n$-crossing} is a crossing where there are $n$ strands meeting at one point, with each strand bisecting the crossing. We call this crossing a \emph{multi-crossing} if $n>2$, and we call the traditional type ($n=2$) a \emph{double crossing}. When $n = 3$ the strands are labeled top, middle and bottom, and when $n \ge 4$ we assign the levels $1, \dots n$ from the top.

In~\cite{triple crossing}, Adams proved that every link has an $n$-crossing projection for all $n \ge 3$. This fact allows us to generalize notions in traditional knot theory to their multi-crossing versions. For example, every link has a \emph{crossing number} $c(L)$, which is the minimum number of crossings in any double crossing projection of the link $L$. We can define the \emph{multi-crossing number} $c_n(L)$ to be the minimum number of crossings in any $n$-crossing projection of the link $L$. This gives us an infinite spectrum of crossing numbers that can be explored.

Let $\beta_2(L)$ denote the braid index of $L$. No et al. give a simple proof of the inequality $\beta_2(L) \le \frac{1}{2}c_2(L) + 1$ using a bisected vertex leveling of a planar graph~\cite{braid index bound}. We generalize this proof to obtain a bound on the triple crossing number in terms of the double crossing braid index. Specifically, we prove the following bound.

\begin{theorem}
\label{thm:braidIndexBound}
Let $L$ be a non-split link. Then
\[
\beta_2(L) \le c_3(L) + 1.
\]
\end{theorem}

In Section~\ref{sec:proof} we present a proof of this result. In Section~\ref{sec:example} we describe an infinite family of knots that achieve equality, and determine the double crossing braid index and the triple crossing number of these knots.

Thanks to Colin Adams for his suggestions on the topic and the approach.

\section{Proof of Theorem~\ref{thm:braidIndexBound}}
\label{sec:proof}

We consider planar graphs embedded in $\R^2$. A bisected vertex leveling of a graph $G$ with $v$ vertices is an rearrangement of the graph between two horizontal lines $y=0$ and $y=v$ such that the following is true:

\begin{enumerate}[topsep=0pt,itemsep=-1ex,partopsep=0ex,parsep=1ex]

\item Each vertex lies on a horizontal strip between $y=k$ and $y=k+1$, for $k=0,1,\dots,v-1$ in which no other vertex lies.
\item Each edge of $G$ has no maxima and minima as critical points of the height function given by the y-coordinate, except its endpoints (vertices).
\item Each line $y = k$, for $k=1,2,\dots,v-1$, cuts $G$ into two pieces, each of which is connected.

\end{enumerate}

No et al. prove that every connected plane graph, which has no monogons or cut vertices, has a bisected vertex leveling~\cite{braid index bound}. Here, a \emph{monogon} is an edge whose two endpoints are the same vertex, and a \emph{cut vertex} is a vertex which, when removed from the graph, splits the graph into two or more connected components, each of which has at least one vertex. In particular, with this definition a monogon is not considered to be a cut vertex. We use this bisected vertex leveling of the projection of a link to obtain a double crossing braid.

We first present a lemma that will lead to the proof of Theorem~\ref{thm:braidIndexBound}.

\begin{lemma}
\label{lem:cutVertices}
Given a triple crossing projection of a link, there is an isotopy so that we obtain a projection with the same number of crossings, such that the following is true.
\begin{itemize}
\item[(i)] There are no cut vertices.
\item[(ii)] For each monogon, the levels of its two ends are top and bottom.
\end{itemize}
\end{lemma}

\begin{proof}
For part (i), suppose a cut vertex appears in a projection. This vertex has degree 6, and each component created by removing this vertex must have an even number of edges entering it. There are two possibilities: we get three components, each with two edges entering it, or we get two components, one with four edges and the other with two edges.

In the first case, the top strand of the crossing may be moved away. Then we are left with a double crossing that is easily removed (Fig.~\ref{fig:bibCut2-2-2}).

\begin{figure}[ht]
	\centering
	\includegraphics[scale=0.25]{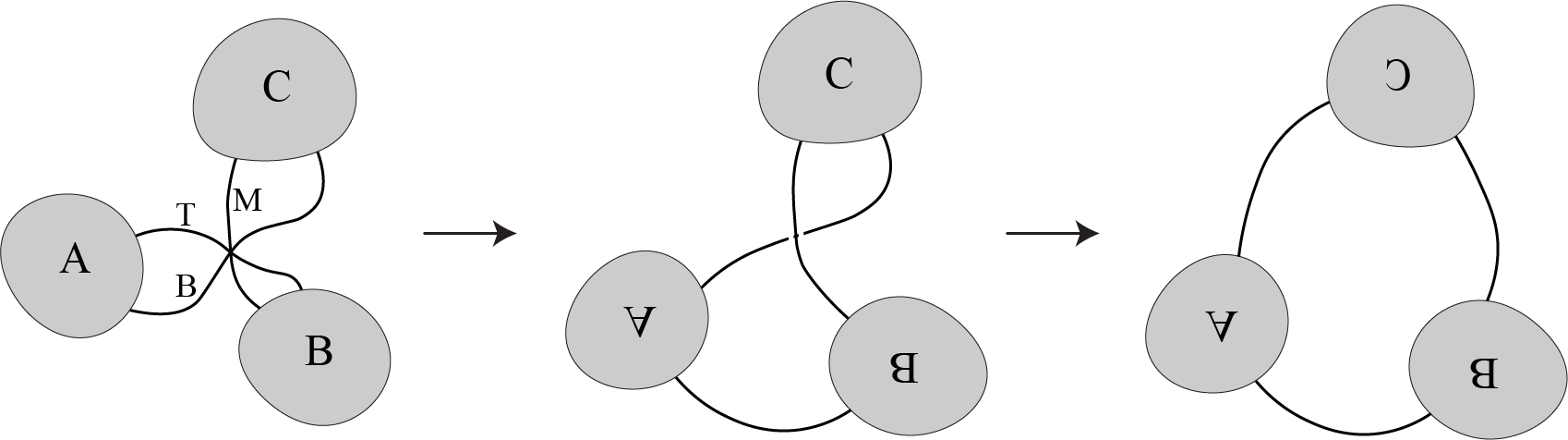}
	\caption{A cut vertex with three components, each with two edges entering it. We can move the top strand away, and then we are left with a double crossing that we can reduce.}
	\label{fig:bibCut2-2-2}
\end{figure}

In the second case, we shift a component as below (Fig.~\ref{fig:bibCut4-2}), so the vertex is no longer a cut vertex, in the sense that the removal of the vertex and the edges adjacent to it will not result in two or more connected components. Note that we now have a monogon instead, which we deal with next.

\begin{figure}[ht]
	\centering
	\includegraphics[scale=0.25]{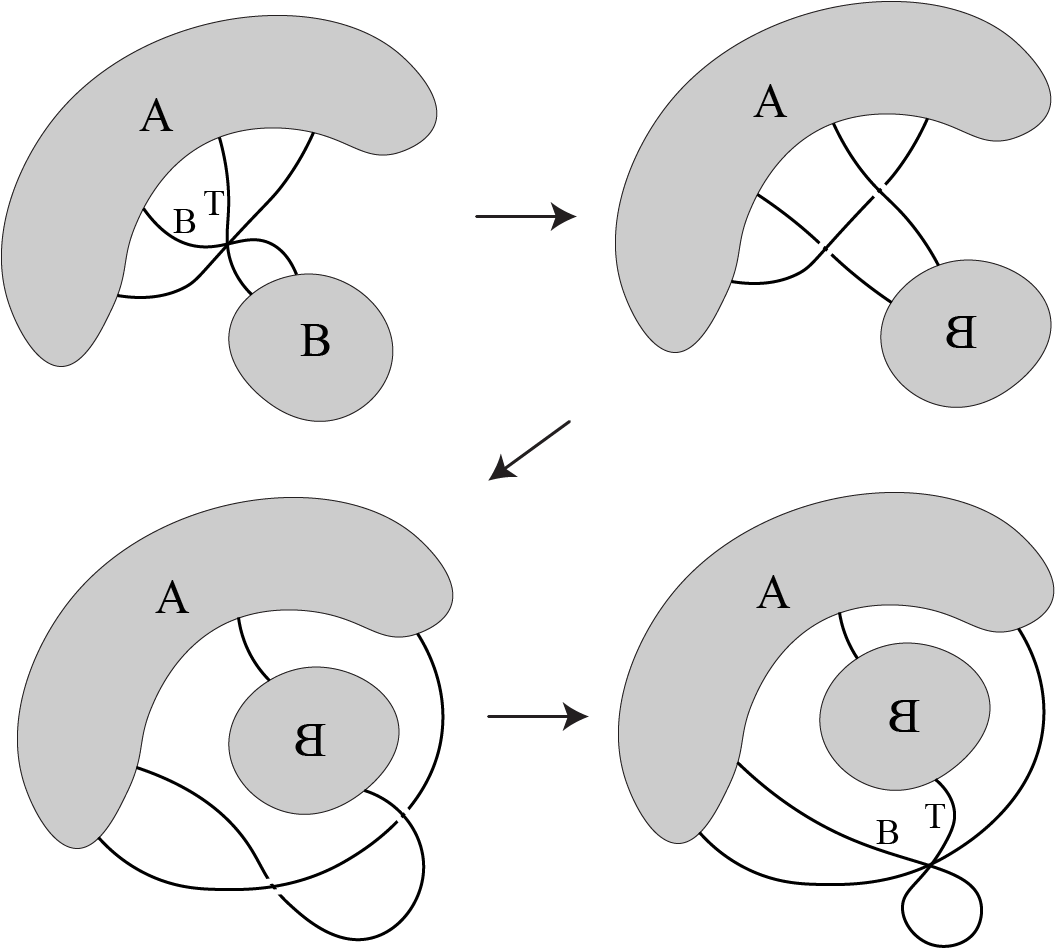}
	\caption{A cut vertex with two components, where one component has four edges entering it and the other has two. We twist the second component to turn the triple crossing into two double crossings. Then we can move this component past the third strand, and recreate the triple crossing. We are now left with a monogon.}
	\label{fig:bibCut4-2}
\end{figure}

For part (ii), take a monogon in a triple crossing projection. Consider the levels of the two ends of the monogon. If the two levels are top and middle, or middle and bottom, we can undo the monogon and remove the crossing, contradicting that the projection was minimal. Therefore the levels must be top and bottom. Note that each vertex can have at most one monogon, for otherwise one of the monogons can be undone.
\end{proof}

Now we are ready to prove Theorem~\ref{thm:braidIndexBound}.

\begin{proof}[Proof of Theorem~\ref{thm:braidIndexBound}]
Consider a minimal triple crossing projection of a non-split link $L$, with $c$ crossings. If we ignore the crossing information, we can consider it to be a planar connected graph $G$, potentially with some monogons and cut vertices. By Lemma~\ref{lem:cutVertices}, we can assume $G$ has no cut vertices, and only monogons where the levels of the two ends are top and bottom.

We then consider a new graph $G'$, which is the graph $G$ minus all monogons. Note that if we can perform an isotopy of $G'$, we can also do the same isotopy on $G$ by first removing the monogons, performing the isotopy, then putting the monogons back where they were. Now $G'$ is a connected graph with no monogons and cut vertices, so we can apply the result of No et al. to obtain a bisected vertex leveling of the graph.

Consider the graph $G$, obtained by adding the monogons back onto $G'$ after a bisected vertex leveling. We know $G$ has $c$ vertices, so $G$ is split into $c$ horizontal strips, each containing a single vertex. Each strip lies between two horizontal lines, and consists of four kinds of arcs: those going from the top line to the vertex, those going from the vertex to the bottom line, those going from the top line to the bottom line, and monogons whose endpoints are both at the vertex (Fig.~\ref{fig:bibStrip}).

\begin{figure}[ht]
	\centering
	\includegraphics[scale=0.3]{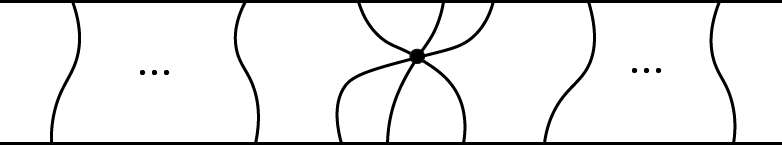}
	\caption{A horizontal strip consists of arcs between the top line, the bottom line, and the vertex.}
	\label{fig:bibStrip}
\end{figure}

Ignore the arcs going from the top line to the bottom line. Then, since each vertex has degree 6, and each vertex has at most one monogon, the strip must look like one of the following or their vertical reflections (Fig.~\ref{fig:bibGraphs}).

\begin{figure}[ht]
	\centering
	\includegraphics[scale=0.3]{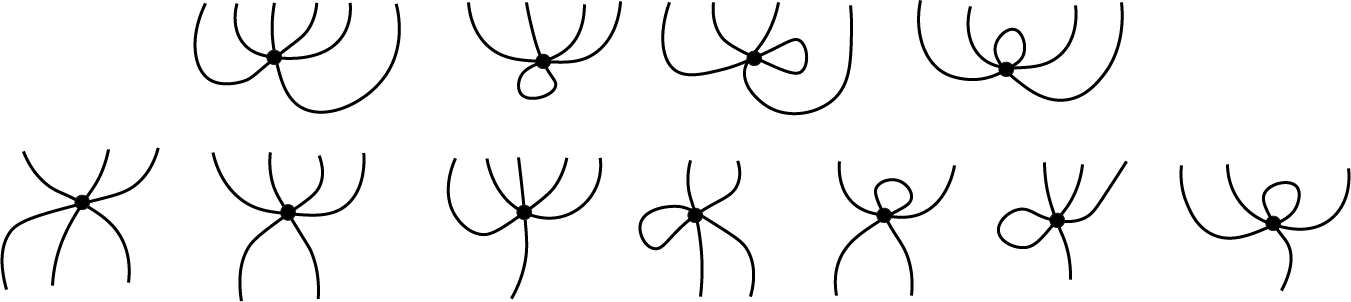}
	\caption{What the arcs of a strip could look like (once the arcs going from the top line to the bottom line are ignored). The top row shows what the strip between $y=0$ and $y=1$ could look like. The strip between $y=v-1$ and $y=v$ will be a vertical reflection of one of the graphs in the top row. All other strips will look like a graph in the bottom row, or its vertical reflection.}
	\label{fig:bibGraphs}
\end{figure}

As in \cite{braid index bound}, we denote by $T_p^q$ a horizontal strip with $p$ arcs going from the bottom line to the vertex, and $q$ arcs going from the vertex to the top line. By the conditions of a bisected vertex leveling, the top strip will be a $T_p^0$ strip for some $p \not = 0$, and the bottom strip will be a $T_0^q$ strip for some $q \not = 0$. The conditions also imply that every strip other than the top and the bottom strips must be a $T_p^q$ for $p,q$ nonzero.

Now restore the top/middle/bottom information at each crossing. We then replace each horizontal strip with a combination of line segments, each with slope 0 (horizontal), 1 (diagonal), or infinity (vertical). Strips with no monogons will be replaced as follows (Fig.~\ref{fig:bibNoLoop}).

\begin{figure}[ht]
	\centering
	\includegraphics[scale=0.3]{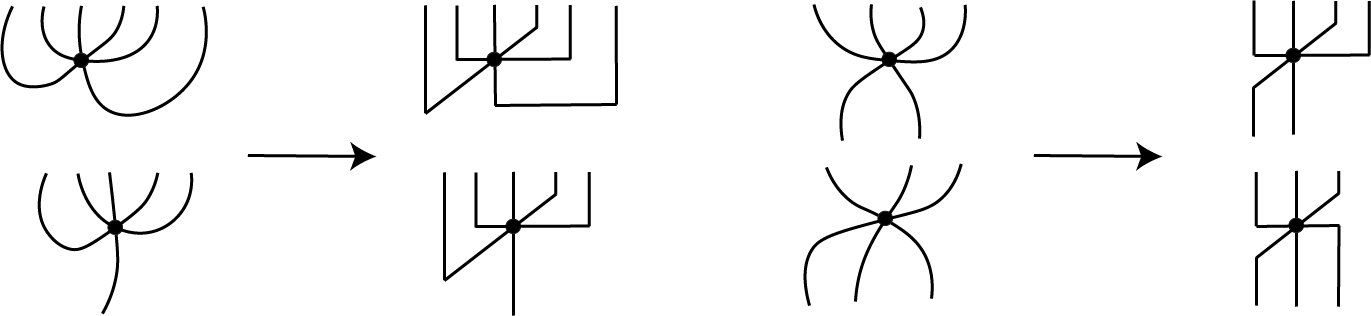}
	\caption{Replacing strips with a combination of horizontal, diagonal, and vertical line segments.}
	\label{fig:bibNoLoop}
\end{figure}

Strips whose vertex has a monogon will first be replaced with similar pictures, but we can reduce them further by undoing the monogons, and rearranging appropriately (Fig.~\ref{fig:bibLoops}). Then each strip can be turned into a combination of horizontal and vertical line segments. Observe that every crossing in this strip is a double crossing and they lie on different levels; in other words, they can be placed on two or three substrips, each of which contains one double crossing.

\begin{figure}[ht]
	\centering
	\includegraphics[scale=0.25]{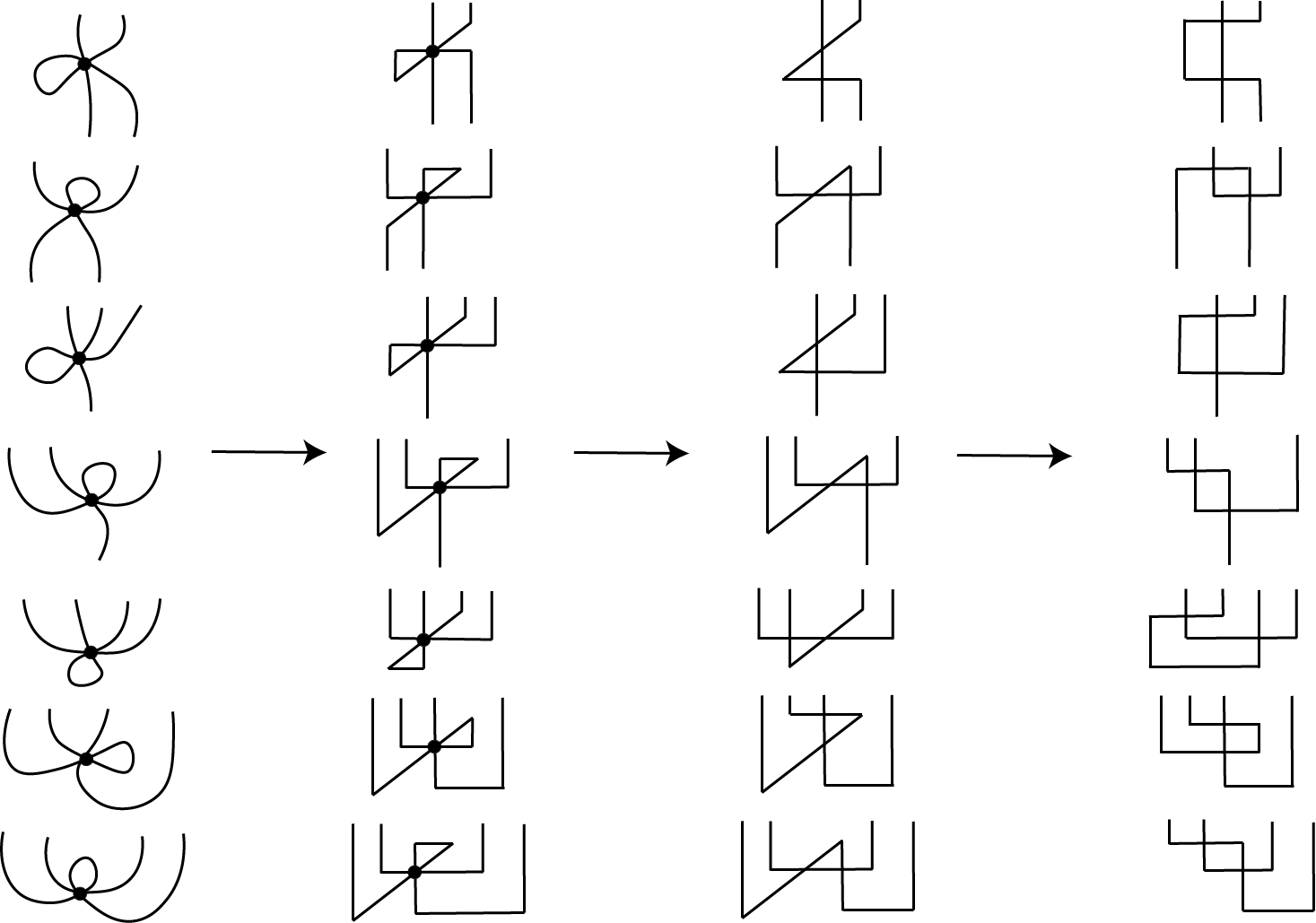}
	\caption{For strips with monogons, we can undo the monogon and rearrange the strings so we are left with horizontal and vertical line segments, where every crossing is a double crossing and lie on different levels.}
	\label{fig:bibLoops}
\end{figure}

Observe that all strips other than $T_0^6$ or $T_0^4$ have exactly two non-vertical line segments, and the strips $T_0^6$ and $T_0^4$ have exactly three non-vertical line segments. This means that we have $2c+2$ non-vertical line segments in the whole graph; 3 each from the top and the bottom strips, and 2 each from the remaining strips.

Choose a natural orientation on the diagram, which is an orientation where at every crossing, the strands alternate between pointing in and out as we go around the crossing. This can be done by checkerboard coloring the complementary regions, and orienting the boundaries of the black regions counterclockwise~\cite{natural orientation}. Each non-vertical line segment has two possibilities, say left directed or right directed. Without loss of generality, we assume that the number of left directed non-vertical line segments are less than or equal to the number of right directed non-vertical line segments.

Our goal is to view this diagram as a braid going from left to right. We describe a procedure to replace certain left directed non-vertical line segments with right directed non-vertical line segments. Suppose we have a left directed non-vertical line segment, which is involved in exactly one crossing as the top (resp. bottom) strand. We can replace it with two external horizontal rays, starting at its two endpoints, that runs through over (resp. under) other line segments (Fig.~\ref{fig:bibSwitch}). We call this a switch move. Note that this is equivalent to pulling the string around the back of the projection and making it right directed. Hence this move is an isotopy that preserves the represented link, once we connect the external rays appropriately. Note also that if the crossing is a double crossing, then both strands of the crossing are either the ``top'' or the ``bottom'', so we can always apply the switch move to double crossings.

\begin{figure}[ht]
	\centering
	\includegraphics[scale=0.3]{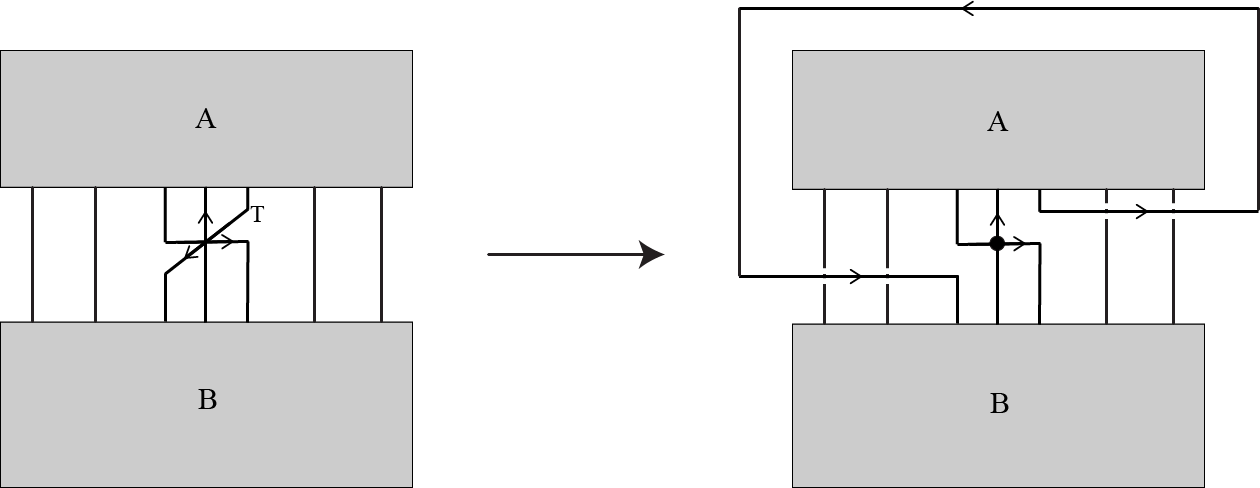}
	\caption{Replacing a left directed segment with two external rays. This can be thought of as pulling a string around the back to make it right directed.}
	\label{fig:bibSwitch}
\end{figure}

Observe that in a strip with a monogon, the crossings are all double crossings that lie on substrips with no other crossings. Hence every line segment is involved in exactly one double crossing, so we can always do the switch move.

For each strip with no monogons, we describe a move so that we no longer have left directed line segments, and the number of additional right directed line segments that goes around the back is at most the number of left directed line segments we started with. We know that each strip has exactly one triple crossing. If the left directed segment is either the top strand or the bottom strand, then we can do the switch move. Hence we consider the cases when we have a middle strand that is left directed. Note that since the orientation is natural, there are two possible choices of orientation to consider for each strip.

If we have a $T_3^3$ strip, we have exactly one left directed segment. If this segment is the middle strand, then we can make this segment vertical (Fig.~\ref{fig:bib3-3}). Then we may introduce a new segment that is left directed, but since this is not the middle strand, this can be pull around to be made right directed.

\begin{figure}[ht]
	\centering
	\includegraphics[scale=0.33]{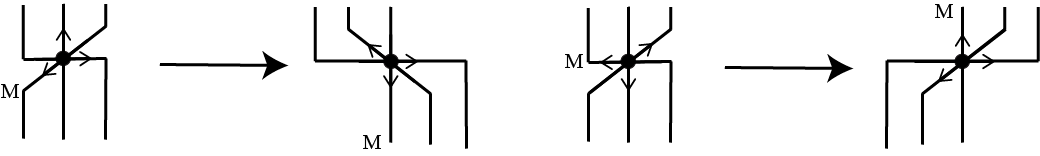}
	\caption{In a $T_3^3$ strip, we can turn any left directed middle strand into a vertical line segment. We can then apply the switch move to any other left directed line segments.}
	\label{fig:bib3-3}
\end{figure}

If we have a $T_2^4$ strip, again we have exactly one left directed segment. If this is the segment that goes from the top line to the bottom line, we can make it vertical as before (Fig.~\ref{fig:bib4-2}). If this is the segment that goes from the top line back to the top line, then we can decompose the triple crossing. We are left with one left directed segment involved in a double crossing, which we can make right directed.

\begin{figure}[ht]
	\centering
	\includegraphics[scale=0.33]{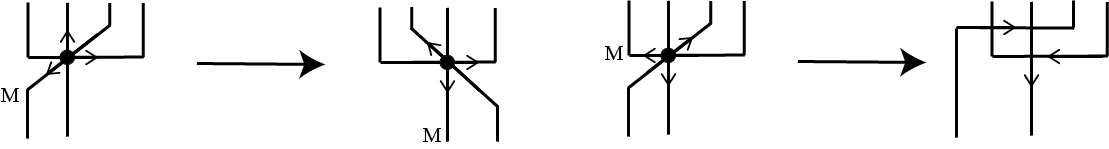}
	\caption{In a $T_2^4$ strip, we can either make the left directed middle strand vertical, or we can decompose the triple crossing so that the left directed line segment is only involved in a double crossing.}
	\label{fig:bib4-2}
\end{figure}

If we have a $T_1^5$ strip, we can apply one of the following moves depending on the orientation (Fig.~\ref{fig:bib5-1}).

\begin{figure}[ht]
	\centering
	\includegraphics[scale=0.33]{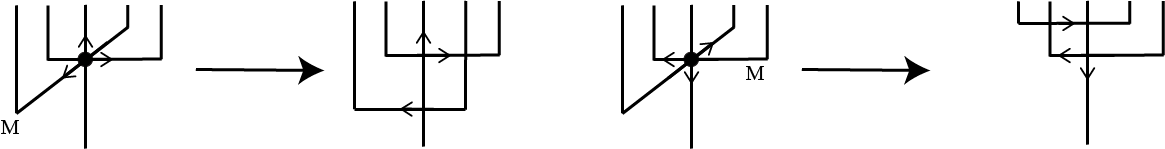}
	\caption{In a $T_1^5$ strip, we can again decompose the triple crossing appropriately so that we are left with one left directed line segment involved in one double crossing.}
	\label{fig:bib5-1}
\end{figure}

If we have a $T_0^6$ strip, then we can either have one or two left directed line segments. In both cases we can obtain a double crossing diagram with the same number of left directed line segment (Fig.~\ref{fig:bib6-0}).

\begin{figure}[!ht]
	\centering
	\includegraphics[scale=0.33]{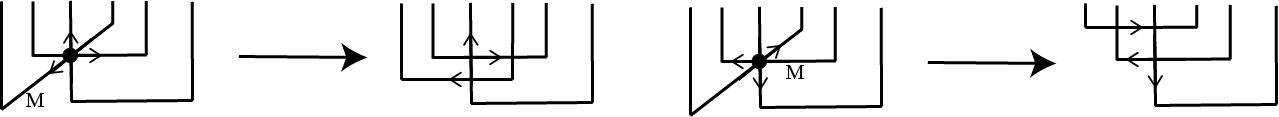}
	\caption{In a $T_0^6$ strip, we may start with one or two left directed line segments, but once the triple crossing is decomposed, we are left with the same number of left directed line segments, each involved in exactly one double crossing.}
	\label{fig:bib6-0}
\end{figure}

Thus we can replace all left directed segments with right directed segments that go around the back. We have replaced them in such a way that the number of additional right directed segments is at most the number of left directed segments that we started with.

This picture can now be viewed as an open braid oriented from left to right. Along with the external rays, connected around the back of the projection, the picture becomes a closed braid which is equivalent to the link we started with. The number of strings in this braid is precisely the number of external rays we added.

Since there were $2c+2$ non-vertical line segments, there were at most $(2c+2)/2 = c+1$ left directed non-vertical line segments. This means that the number of external rays we added was at most $c+1$. Hence we have a braid representation for $L$ with at most $c+1$ strings, which shows that the braid index of $L$ is at most $c+1$.
\end{proof}

\section{Infinite family of knots that achieve equality}
\label{sec:example}

In this section we describe an infinite family of knots for which we can determine their braid indices and their triple crossing numbers using this new bound.

The new bound of $\beta_2(L) \le c_3(L) + 1$ improves on the previous bound of $\beta_2 \le \frac{1}{2} c_2(L)+1$ for links with $c_3(L) < \frac{1}{2} c_2(L)$. Note that for all alternating links, we have $c_3(L) \ge \frac{1}{2} c_2(L)$ as shown in \cite{triple crossing}, so our example must be a non-alternating link. One example of such a link is the following 12-crossing knot with $c_3(L) = 4$ (Fig.~\ref{fig:12knot}).

\begin{figure}[!ht]
	\centering
	\includegraphics[scale=0.2]{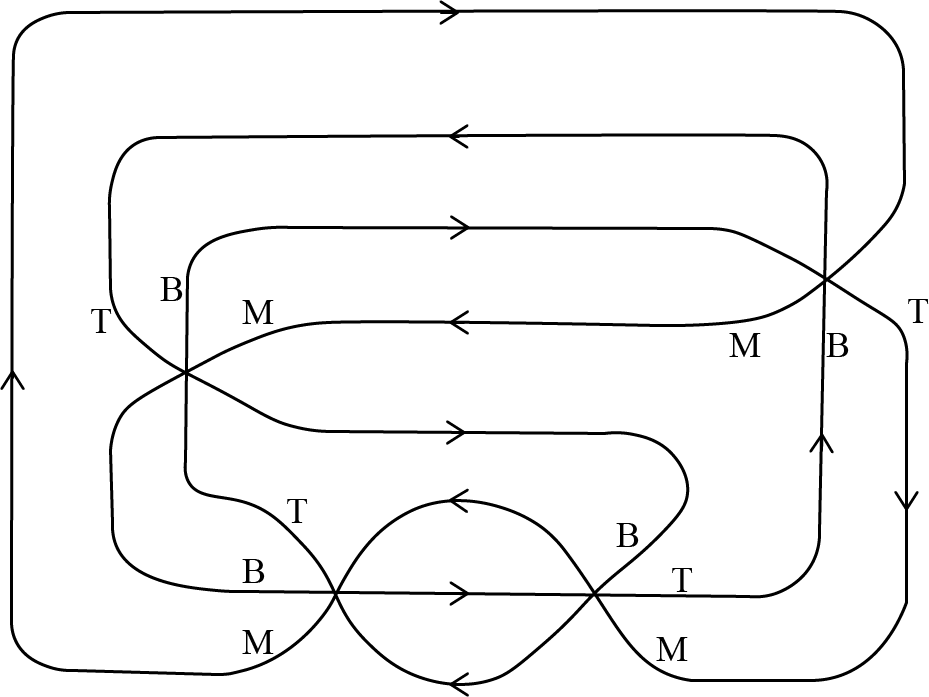}
	\caption{An example of a knot with $c_2(L) = 3c_3(L)$.}
	\label{fig:12knot}
\end{figure}

We consider a family of knots where we repeat the two triple crossings in the $A$ portion indicated in Fig.~\ref{fig:12more}.

\begin{figure}[!ht]
	\centering
	\includegraphics[scale=0.2]{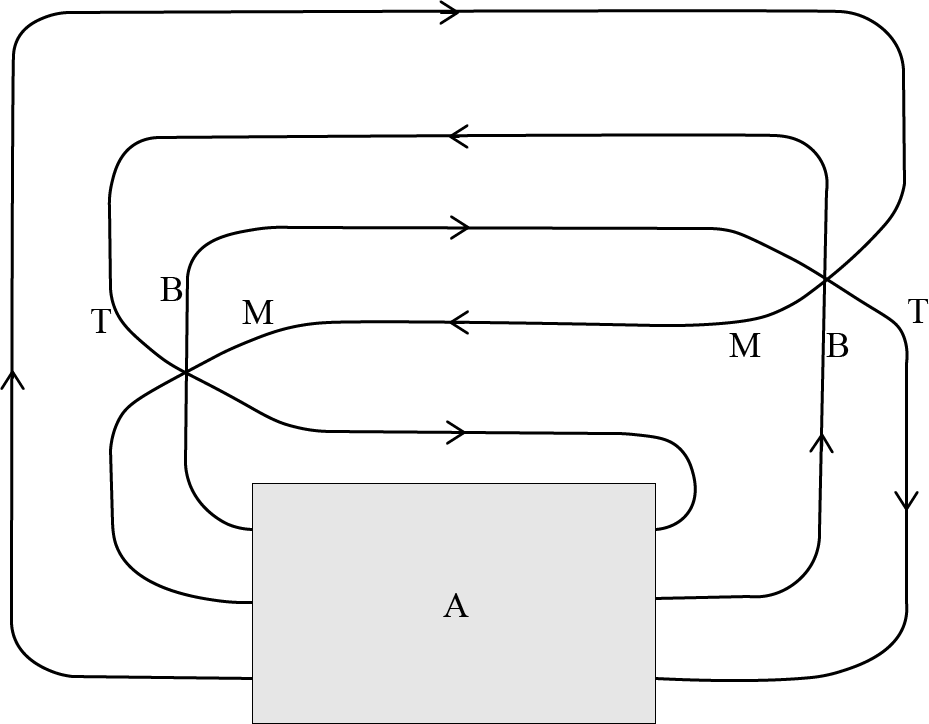}
	\caption{Basic structure which will be used to construct the infinite family of knots.}
	\label{fig:12more}
\end{figure}

In other words, we replace the $A$ portion with a sequence of pairs of triple crossings in Fig.\ref{fig:12tbmbtm}.

\begin{figure}[!ht]
	\centering
	\includegraphics[scale=0.2]{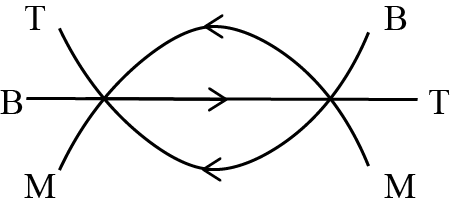}
	\caption{A pair of triple crossings that will be repeated in the $A$ portion of Fig.~\ref{fig:12more}.}
	\label{fig:12tbmbtm}
\end{figure}

We determine the braid indices using the HOMFLY polynomial bound. The HOMFLY polynomial is given by the skein relation
\[
v^{-1} P\bigg( \raisebox{-8pt}{\includegraphics[scale=0.15]{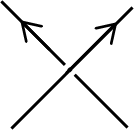}} \bigg)
- v P\bigg( \raisebox{-8pt}{\includegraphics[scale=0.15]{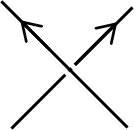}} \bigg)
= z P\bigg( \raisebox{-10pt}{\includegraphics[scale=0.15]{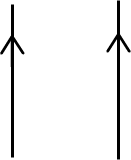}} \bigg),
\]
along with the condition that the polynomial for the unknot is 1. Morton, Franks and Williams proved the following bound on braid indices~\cite{franks williams}~\cite{morton}:
\[
\beta_2(L) \ge \frac{v\text{-span of } P(L) }{2} + 1.
\]

We have the following result.

\begin{theorem}
Let $L$ be the knot obtained by replacing the $A$ portion of Fig.~\ref{fig:12more} with $n$ pairs of triple crossings in Fig.~\ref{fig:12tbmbtm}. Then $\beta_2(L) = 2n+3$, and $c_3(L) = 2n+2$.
\end{theorem}

\begin{proof}

Since we are only interested in the $v$-span of the HOMFLY polynomial, we consider equivalence classes of HOMFLY polynomials given by their span. Specifically, we say $P(L_1) = P(L_2)$ if the $v$-span of the two polynomials are the same.

For this example, the highest and the lowest powers of $v$ only appear once, so there is no danger of cancelling. Therefore we do not keep track of the signs of powers of $v$, or any powers of $z$. Then the skein relation can be rewritten as follows:
\begin{align*}
P\bigg( \raisebox{-8pt}{\includegraphics[scale=0.15]{12skeinA.png}} \bigg)
&= v^2 P\bigg( \raisebox{-8pt}{\includegraphics[scale=0.15]{12skeinB.png}} \bigg)
+ v P\bigg( \raisebox{-10pt}{\includegraphics[scale=0.15]{12skeinC.png}} \bigg), \\
P\bigg( \raisebox{-8pt}{\includegraphics[scale=0.15]{12skeinB.png}} \bigg)
&= v^{-2} P\bigg( \raisebox{-8pt}{\includegraphics[scale=0.15]{12skeinA.png}} \bigg)
+ v^{-1} P\bigg( \raisebox{-10pt}{\includegraphics[scale=0.15]{12skeinC.png}} \bigg). \\
\end{align*}

Using these relations, we have the following.

\begin{align*}
&P\bigg( \raisebox{-15pt}{\includegraphics[scale=0.15]{12tbmbtm.png}} \bigg)
= v^2 P\bigg( \raisebox{-15pt}{\includegraphics[scale=0.15]{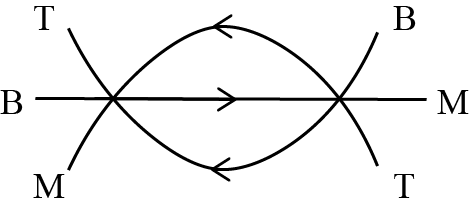}} \bigg)
+ v P\bigg( \raisebox{-15pt}{\includegraphics[scale=0.15]{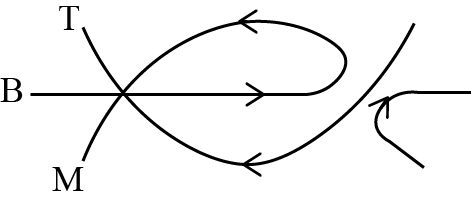}} \bigg) \\
&= v^2 P\bigg( \raisebox{-15pt}{\includegraphics[scale=0.15]{12tbmbmt.png}} \bigg)
+ v P\bigg( \raisebox{-12pt}{\includegraphics[scale=0.15]{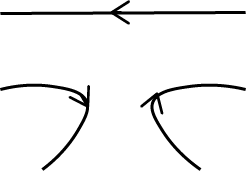}} \bigg) \\
&= v^2 P\bigg( \raisebox{-12pt}{\includegraphics[scale=0.15]{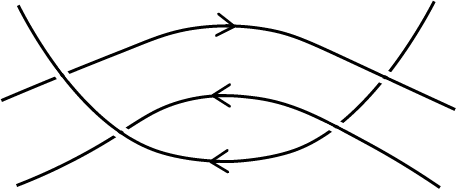}} \bigg)
+ v P\bigg( \raisebox{-12pt}{\includegraphics[scale=0.15]{12lowlow.png}} \bigg) \\
&= v^2 \Bigg( v^2 P\bigg( \raisebox{-12pt}{\includegraphics[scale=0.15]{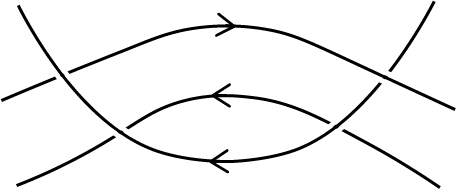}} \bigg)
+ v  P\bigg( \raisebox{-12pt}{\includegraphics[scale=0.15]{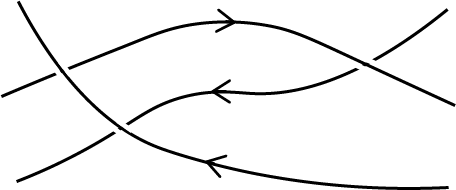}} \bigg) \Bigg)
+ v P\bigg( \raisebox{-12pt}{\includegraphics[scale=0.15]{12lowlow.png}} \bigg) \\
&= v^4 P\bigg( \raisebox{-10pt}{\includegraphics[scale=0.15]{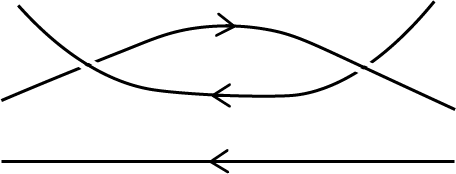}} \bigg)
+ v^3  P\bigg( \raisebox{-15pt}{\includegraphics[scale=0.15]{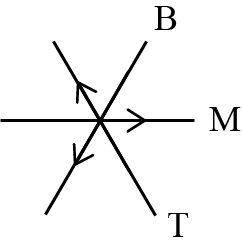}} \bigg)
+ v P\bigg( \raisebox{-12pt}{\includegraphics[scale=0.15]{12lowlow.png}} \bigg) \\
&= v^4 \Bigg( v^{-2} P\bigg( \raisebox{-10pt}{\includegraphics[scale=0.15]{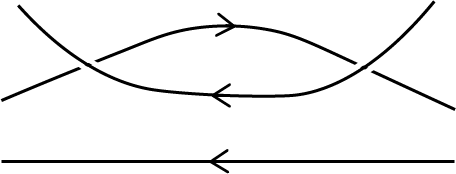}} \bigg)
+ v^{-1} P\bigg( \raisebox{-10pt}{\includegraphics[scale=0.15]{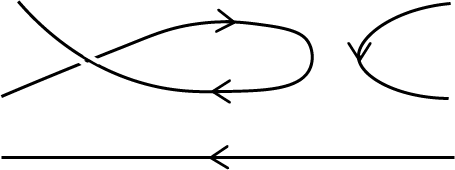}} \bigg) \Bigg)
+ v^3  P\bigg( \raisebox{-15pt}{\includegraphics[scale=0.15]{12bmt.png}} \bigg)
+ v P\bigg( \raisebox{-12pt}{\includegraphics[scale=0.15]{12lowlow.png}} \bigg) \\
&= v P\bigg( \raisebox{-12pt}{\includegraphics[scale=0.15]{12lowlow.png}} \bigg)
+ v^3  P\bigg( \raisebox{-15pt}{\includegraphics[scale=0.15]{12bmt.png}} \bigg)
+ v^3 P\bigg( \raisebox{-10pt}{\includegraphics[scale=0.13]{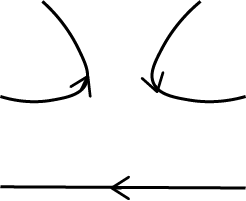}} \bigg)
+ v^2 P\bigg( \raisebox{-8pt}{\includegraphics[scale=0.15]{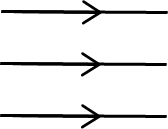}} \bigg)
\end{align*}

We can also see that adding a trivial component is a multiplication by $(v+v^{-1})$:

\begin{align*}
P\bigg( \raisebox{-8pt}{\includegraphics[scale=0.1]{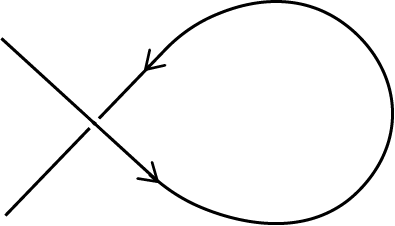}} \bigg)
&= v^2 P\bigg( \raisebox{-8pt}{\includegraphics[scale=0.1]{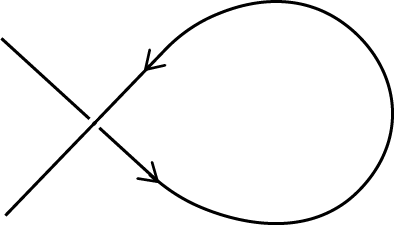}} \bigg)
+ v P\bigg( \raisebox{-9pt}{\includegraphics[scale=0.1]{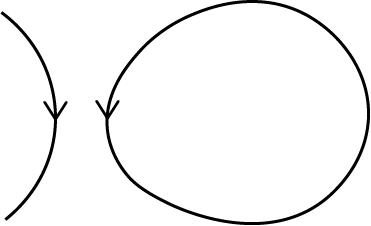}} \bigg) \\
P\bigg( \raisebox{-9pt}{\includegraphics[scale=0.1]{12unlinkC.png}} \bigg)
&= (v + v^{-1}) P\bigg( \hspace{3mm} \raisebox{-8pt}{\includegraphics[scale=0.1]{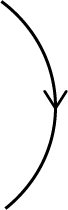}} \hspace{3mm} \bigg).
\end{align*}

For a diagram $T$, let $P_i(T)$ be the equivalence class of the HOMFLY polynomial for the link given by replacing the $A$ portion of Fig.~\ref{fig:12more} with $i$ triple crossings (on the left) and $T$ (on the right).

Using these relations, we can derive the following.

\begin{align*}
&P_{i+1}\bigg( \raisebox{-10pt}{\includegraphics[scale=0.13]{12highhigh.png}} \bigg) 
= P_i\bigg( \raisebox{-15pt}{\includegraphics[scale=0.14]{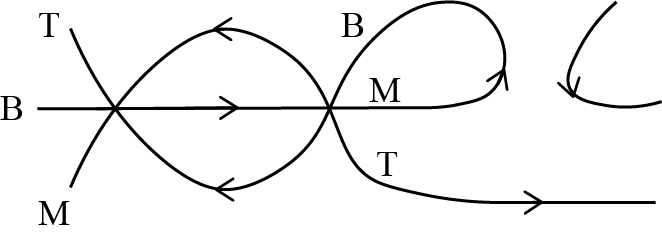}} \bigg) \\
&= v P_i\bigg( \raisebox{-12pt}{\includegraphics[scale=0.15]{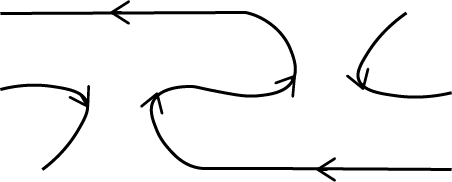}} \bigg)
+ v^3 P_i\bigg( \raisebox{-15pt}{\includegraphics[scale=0.15]{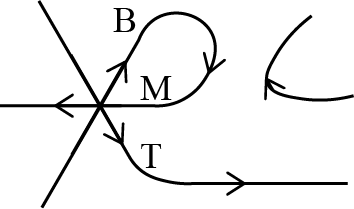}} \bigg)
+ v^3 P_i\bigg( \raisebox{-10pt}{\includegraphics[scale=0.15]{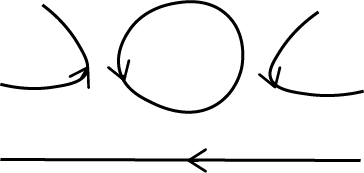}} \bigg)
+ v^2 P_i\bigg( \raisebox{-10pt}{\includegraphics[scale=0.13]{12highhigh.png}} \bigg) \\
&= v P_i\bigg( \raisebox{-12pt}{\includegraphics[scale=0.12]{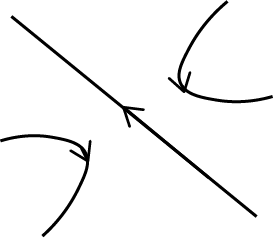}} \bigg)
+ v^3 P_i\bigg( \raisebox{-12pt}{\includegraphics[scale=0.12]{12lowhigh.png}} \bigg)
+ v^3 P_i\bigg( \raisebox{-10pt}{\includegraphics[scale=0.15]{12highhighhighhigh.png}} \bigg)
+ v^2 P_i\bigg( \raisebox{-10pt}{\includegraphics[scale=0.13]{12highhigh.png}} \bigg). \\
&= (v + v^3) P_i\bigg( \raisebox{-12pt}{\includegraphics[scale=0.12]{12lowhigh.png}} \bigg)
+ (v^2 + v^4) P_i\bigg( \raisebox{-10pt}{\includegraphics[scale=0.13]{12highhigh.png}} \bigg) \\
&P_{i+1}\bigg( \raisebox{-12pt}{\includegraphics[scale=0.12]{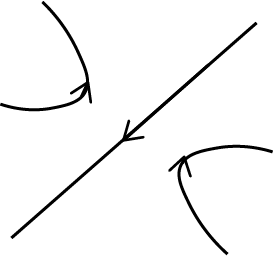}} \bigg)
= P_i\bigg( \raisebox{-15pt}{\includegraphics[scale=0.15]{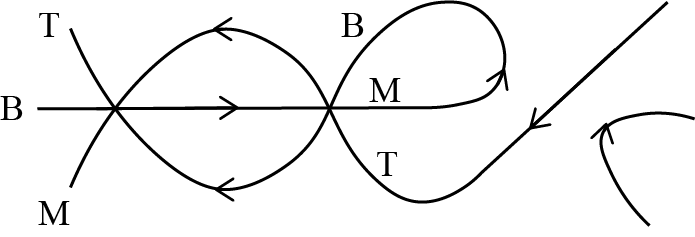}} \bigg) \\
&= v P_i\bigg( \raisebox{-12pt}{\includegraphics[scale=0.15]{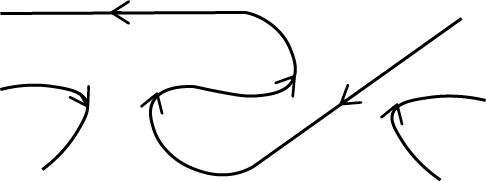}} \bigg)
+ v^3 P_i\bigg( \raisebox{-15pt}{\includegraphics[scale=0.15]{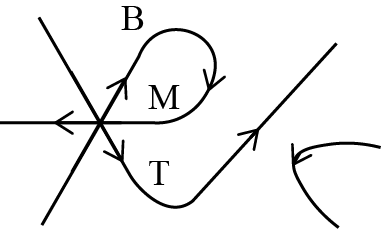}} \bigg)
+ v^3 P_i\bigg( \raisebox{-12pt}{\includegraphics[scale=0.12]{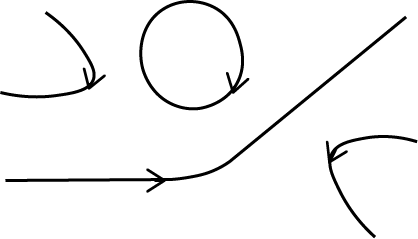}} \bigg)
+ v^2 P_i\bigg( \raisebox{-12pt}{\includegraphics[scale=0.12]{12highlow.png}} \bigg) \\
&= v P_i\bigg( \raisebox{-12pt}{\includegraphics[scale=0.15]{12lowlow.png}} \bigg)
+ v^3 P_i\bigg( \raisebox{-12pt}{\includegraphics[scale=0.15]{12lowlow.png}} \bigg)
+ v^3 P_i\bigg( \raisebox{-12pt}{\includegraphics[scale=0.12]{12highhighhighlow.png}} \bigg)
+ v^2 P_i\bigg( \raisebox{-12pt}{\includegraphics[scale=0.12]{12highlow.png}} \bigg). \\
&= (v+v^3) P_i\bigg( \raisebox{-12pt}{\includegraphics[scale=0.15]{12lowlow.png}} \bigg)
+ (v^2 + v^4) P_i\bigg( \raisebox{-12pt}{\includegraphics[scale=0.12]{12highlow.png}} \bigg). \\
&P_{i+1}\bigg( \raisebox{-15pt}{\includegraphics[scale=0.15]{12bmt.png}} \bigg)
= P_i\bigg( \raisebox{-15pt}{\includegraphics[scale=0.15]{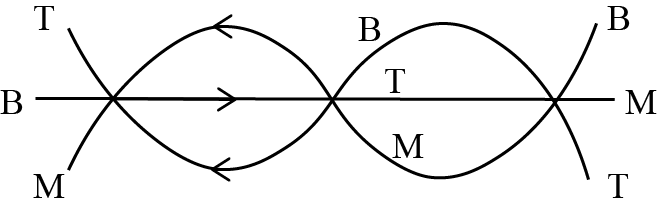}} \bigg) \\
&= v P_i\bigg( \raisebox{-15pt}{\includegraphics[scale=0.12]{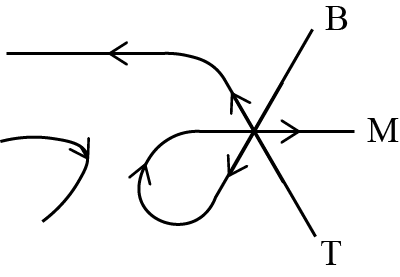}} \bigg)
+ v^3 P_i\bigg( \raisebox{-15pt}{\includegraphics[scale=0.15]{12tbmbmt.png}} \bigg)
+ v^3 P_i\bigg( \raisebox{-15pt}{\includegraphics[scale=0.12]{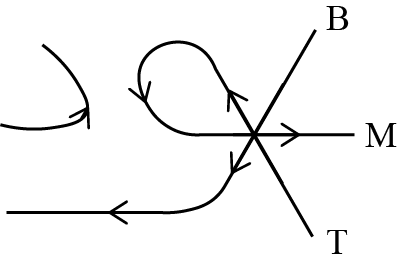}} \bigg)
+ v^2 P_i\bigg( \raisebox{-15pt}{\includegraphics[scale=0.15]{12bmt.png}} \bigg) \\
&= v P_i\bigg( \raisebox{-12pt}{\includegraphics[scale=0.12]{12lowhigh.png}} \bigg)
+ v^3 P_i\bigg( \raisebox{-15pt}{\includegraphics[scale=0.15]{12tbmbmt.png}} \bigg)
+ v^3 P_i\bigg( \raisebox{-12pt}{\includegraphics[scale=0.12]{12highlow.png}} \bigg)
+ v^2 P_i\bigg( \raisebox{-15pt}{\includegraphics[scale=0.15]{12bmt.png}} \bigg).
\end{align*}

In addition,

\begin{align*}
&P_i\bigg( \raisebox{-15pt}{\includegraphics[scale=0.15]{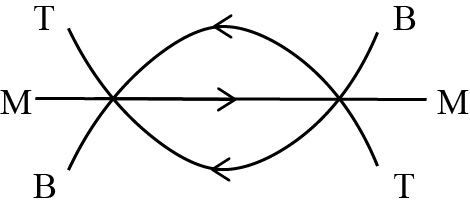}} \bigg)
= v^{-2} P_i\bigg( \raisebox{-15pt}{\includegraphics[scale=0.15]{12tbmbmt.png}} \bigg)
+ v^{-1} P_i\bigg( \raisebox{-15pt}{\includegraphics[scale=0.15]{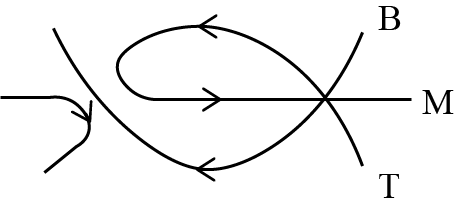}} \bigg) \\
&= v^{-2} \Bigg( v^{-2} P_i\bigg( \raisebox{-15pt}{\includegraphics[scale=0.15]{12tbmbtm.png}} \bigg)
+ v^{-1} P \bigg( \raisebox{-15pt}{\includegraphics[scale=0.15]{12tbmbtm2.png}} \bigg) \Bigg)
+ v^{-1} P_i\bigg( \raisebox{-12pt}{\includegraphics[scale=0.15]{12lowlow.png}} \bigg) \\
&= v^{-4} P_i\bigg( \raisebox{-15pt}{\includegraphics[scale=0.15]{12tbmbtm.png}} \bigg)
+ (v^{-3} + v^{-1}) P_i\bigg( \raisebox{-12pt}{\includegraphics[scale=0.15]{12lowlow.png}} \bigg).
\end{align*}

So we have

\begin{align*}
&P_{i+1}\bigg( \raisebox{-15pt}{\includegraphics[scale=0.15]{12bmt.png}} \bigg)
= v P_i\bigg( \raisebox{-12pt}{\includegraphics[scale=0.12]{12lowhigh.png}} \bigg)
+ v^{-1} P_i\bigg( \raisebox{-15pt}{\includegraphics[scale=0.15]{12tbmbtm.png}} \bigg) \\
&\hspace{5mm}+ (1 + v^2) P_i\bigg( \raisebox{-12pt}{\includegraphics[scale=0.15]{12lowlow.png}} \bigg)
+ v^3 P_i\bigg( \raisebox{-12pt}{\includegraphics[scale=0.12]{12highlow.png}} \bigg)
+ v^2 P_i\bigg( \raisebox{-15pt}{\includegraphics[scale=0.15]{12bmt.png}} \bigg).
\end{align*}

Consider the pair of triple crossings in Fig.~\ref{fig:12tbmbtm}. If we join the top and the middle strands of the crossing on the right and create a loop, then the loop can be pulled over the crossing to remove the crossing on this crossing. Observe that the loop now connects the bottom and the middle strands of the left crossing. Then the loop can be pulled under this crossing to remove it. In other words, the pair of triple crossings can be removed if we connect the top and the middle strands of the crossing on the right. Thus we can see that 
\begin{align*}
&P_i\bigg( \raisebox{-12pt}{\includegraphics[scale=0.15]{12lowlow.png}} \bigg)
= P_0\bigg( \raisebox{-12pt}{\includegraphics[scale=0.15]{12lowlow.png}} \bigg)
= v^{-5} + v^{-1}, \\
&P_i\bigg( \raisebox{-12pt}{\includegraphics[scale=0.12]{12lowhigh.png}} \bigg)
= P_0\bigg( \raisebox{-12pt}{\includegraphics[scale=0.12]{12lowhigh.png}} \bigg)
= v^{-4} + 1.
\end{align*}
The $P_0$ polynomials have been calculated by a program.

We claim that the following is true, and prove this by induction.

\begin{align*}
P_n\bigg( \raisebox{-10pt}{\includegraphics[scale=0.13]{12highhigh.png}} \bigg)
&= \begin{cases}
v^{-3} + v^{4n-3} &\text{ if } n \ge 2; \\
v^{-1} + v &\text{ if } n = 1.
\end{cases}\\
P_n\bigg( \raisebox{-12pt}{\includegraphics[scale=0.12]{12highlow.png}} \bigg)
&= \begin{cases}
v^{-4} + v^{4n-4} &\text{ if } n \ge 2; \\
1 &\text{ if } n = 1.
\end{cases}\\
P_n\bigg( \raisebox{-15pt}{\includegraphics[scale=0.15]{12bmt.png}} \bigg)
&= v^{-5} + v^{4n-5} \\
P_n\bigg( \raisebox{-15pt}{\includegraphics[scale=0.15]{12tbmbtm.png}} \bigg)
&=v^{-4} + v^{4n}.
\end{align*}

The cases for $n=1,2$ can be verified by a program. For the inductive step, suppose the equations for $n=i$ are true, where $i \ge 2$. Then we have

\begin{align*}
P_{i+1}\bigg( \raisebox{-10pt}{\includegraphics[scale=0.13]{12highhigh.png}} \bigg)
&= (v + v^3) P_i\bigg( \raisebox{-12pt}{\includegraphics[scale=0.12]{12lowhigh.png}} \bigg)
+ (v^2 + v^4) P_i\bigg( \raisebox{-10pt}{\includegraphics[scale=0.13]{12highhigh.png}} \bigg) \\
&= (v+v^3)(v^{-4}+1) + (v^2 + v^4)(v^{-3} + v^{4i-3}) \\
&= v^{-3} + v^{4(i+1)-3}. \\
P_{i+1}\bigg( \raisebox{-12pt}{\includegraphics[scale=0.12]{12highlow.png}} \bigg)
&= (v+v^3) P_i\bigg( \raisebox{-12pt}{\includegraphics[scale=0.15]{12lowlow.png}} \bigg)
+ (v^2 + v^4) P_i\bigg( \raisebox{-12pt}{\includegraphics[scale=0.12]{12highlow.png}} \bigg). \\
&= (v + v^3)(v^{-5} + v^{-1}) + (v^2 + v^4)(v^{-4} + v^{4i-4}) \\
&= v^{-4} + v^{4(i+1) - 4}. \\
P_{i+1}\bigg( \raisebox{-15pt}{\includegraphics[scale=0.15]{12bmt.png}} \bigg)
&= v P_i\bigg( \raisebox{-12pt}{\includegraphics[scale=0.12]{12lowhigh.png}} \bigg)
+ v^{-1} P_i\bigg( \raisebox{-15pt}{\includegraphics[scale=0.15]{12tbmbtm.png}} \bigg)
+ (1 + v^2) P_i\bigg( \raisebox{-12pt}{\includegraphics[scale=0.15]{12lowlow.png}} \bigg) \\
&\hspace{5mm}+ v^3 P_i\bigg( \raisebox{-12pt}{\includegraphics[scale=0.12]{12highlow.png}} \bigg)
+ v^2 P_i\bigg( \raisebox{-15pt}{\includegraphics[scale=0.15]{12bmt.png}} \bigg) \\
&= v(v^{-4} + 1) + v^{-1}(v^{-4} + v^{4i}) + (1+v^2)(v^{-5} + v^{-1}) \\
&\hspace{5mm}+ v^3(v^{-4} + v^{4i-4}) + v^2(v^{-5} + v^{4i-5}) \\
&= v^{-5} + v^{4(i+1)-5}. \\
P_{i+1}\bigg( \raisebox{-15pt}{\includegraphics[scale=0.15]{12tbmbtm.png}} \bigg)
&= v P_{i+1}\bigg( \raisebox{-12pt}{\includegraphics[scale=0.15]{12lowlow.png}} \bigg)
+ v^3 P_{i+1}\bigg( \raisebox{-15pt}{\includegraphics[scale=0.15]{12bmt.png}} \bigg) \\
&\hspace{5mm}+ v^3 P_{i+1}\bigg( \raisebox{-10pt}{\includegraphics[scale=0.13]{12highhigh.png}} \bigg)
+ v^2 P_{i+1}\bigg( \raisebox{-8pt}{\includegraphics[scale=0.15]{12none.png}} \bigg) \\
&= v P_{i+1}\bigg( \raisebox{-12pt}{\includegraphics[scale=0.15]{12lowlow.png}} \bigg)
+ v^3 P_{i+1}\bigg( \raisebox{-15pt}{\includegraphics[scale=0.15]{12bmt.png}} \bigg) \\
&\hspace{5mm}+ v^3 P_{i+1}\bigg( \raisebox{-10pt}{\includegraphics[scale=0.13]{12highhigh.png}} \bigg)
+ v^2 P_i\bigg( \raisebox{-15pt}{\includegraphics[scale=0.15]{12tbmbtm.png}} \bigg) \\
&= v(v^{-5} + v^{-1}) + v^3(v^{-5} + v^{4(i+1)-5}) + v^3(v^{-3} + v^{4(i+1)-3}) + v^2(v^{-4} + v^{4i}) \\
&= v^{-4} + v^{4(i+1)}.
\end{align*}

This proves the claim. Therefore, we have $P(L) = v^{-4} + v^{4n}$, so the $v$-span of $L$ is $4n+4$. We know that $L$ has a triple crossing projection with $2n+2$ crossings. This means that
\[
\frac{4n+4}{2} + 1 = \frac{v\text{-span of } P(L) }{2} + 1 \le \beta_2(L) \le c_3(L) + 1 \le (2n+2) + 1,
\]
which shows that $\beta_2(L) = c_3(L) + 1 = 2n+3$.
\end{proof}

We expect that these projections of this family of knots realize the minimal double crossing number, meaning this would be an example of an infinite family of knots such that $c_2(L) = 3c_3(L)$. Unfortunately, we do not have a means to prove that $c_2(L)$ cannot be lower.

\FloatBarrier 

\end{document}